\documentclass[12pt]{amsart}
\usepackage{latexsym,amscd,amssymb, comment}  
\numberwithin{equation}{section}

\theoremstyle{plain} 
\newtheorem{thm}{Theorem}[section]
\newtheorem{cor}[thm]{Corollary}
\newtheorem{lem}[thm]{Lemma}
\newtheorem{prop}[thm]{Proposition}

\theoremstyle{definition}
\newtheorem{defn}[thm]{Definition}
\newtheorem{ex}[thm]{Example}

\theoremstyle{remark}
\newtheorem{rem}[thm]{Remark}

\newcommand\kk{\Bbbk}
\def\NN{\mathbb N}
\def\ZZ{\mathbb Z}
\def\Z+{{\mathbb Z}_{>0}}

\def\cJ{\mathcal J}

\def\0{\mathbf 0}
\def\ba{\mathbf a}
\def\bb{\mathbf b}
\def\bc{\mathbf c}

\def\be{\mathbf e}
\def\bba{\overline{\mathbf a}}
\def\one{\mathbf 1}
\def\wba{\widetilde{\mathbf a}}
\def\oba{\widehat{\mathbf a}}
\def\fm{\mathfrak m}

\def\wfm{\widetilde{{\mathfrak m}}}

\def\<{\langle}
\def\>{\rangle}
\def\wI{\widetilde{I}}
\def\oI{\overline{I}}
\def\wP{\widetilde{P}}
\def\wS{\widetilde{S}}
\def\wX{\widetilde{X}}

\def\sm{{\mathsf m}}
\def\sn{{\mathsf n}}
\def\st{{\mathsf t}}

\def\bpol{\operatorname{b-pol}}
\def\Hom{\operatorname{Hom}}
\def\Ext{\operatorname{Ext}}
\def\Tor{\operatorname{Tor}}
\def\height{\operatorname{ht}}
\def\pol{\operatorname{pol}}
\def\adeg{\operatorname{adeg}}
\def\pd{\operatorname{proj-dim}}

\def\hba{\widehat{\mathbf a}}

\def\wbb{\widetilde{\mathbf b}}

\begin{document}

\title[{Alexander duality for strongly stable ideals}]{Alexander duality for the alternative polarizations of strongly stable ideals}
\author{Kosuke Shibata}
\address{Department of Mathematics, Okayama University, Okayama, Okayama 700-8530, Japan}
\email{pfel97d6@s.okayama-u.ac.jp}
\author{Kohji Yanagawa}
\address{Department of Mathematics, Kansai University, Suita, Osaka 564-8680, Japan}
\email{yanagawa@kansai-u.ac.jp}
\thanks{The second author is partially supported by JSPS Grant-in-Aid for Scientific Research (C) 19K03456.}
\maketitle
\begin{abstract}
%Using the technique of {\it alternative polarization}, 
We will define the Alexander duality for strongly stable ideals.  More precisely, for a strongly stable ideal $I \subset \kk[x_1, \ldots, x_n]$ with $\deg(\sm) \le d$ for all $\sm \in G(I)$, its dual $I^* \subset \kk[y_1, \ldots, y_d]$ is a strongly stable ideal with  $\deg(\sm) \le n$ for all $\sm \in G(I^*)$. 
%An expected relation between the Hilbert series of the local cohomologies $H_\fm^i(S/I)$ and the graded Betti numbers of $I^*$ holds. 
This duality has been constructed by Fl$\o$ystad et al.\! in a different manner, so we emphasis applications here. For example, we will describe the Hilbert series of the local cohomologies $H_\fm^i(S/I)$ using the irreducible decomposition of  $I$ (through the Betti numbers of $I^*$).
  %vis the Eliahou-Keraire formula. 
%For a strongly stable ideal $I$, we show that the Alexander dual of $\bpol(I)$ is $\bpol(I^*)$ of some strongly stable ideal $I^*$ after switching the variables.  It yields that Hilbert series of $H_{\fm}^i(S/I)$ can be described by graded Betti number of $I^*$. Moreover 
%We also show that if $S/I$ is Cohen-Macaulay then $\bpol(I)$ is a letterplace ideal in the sense of  Fl$\o$ystad.We show that the irreducible components of $\bpol(I)$ is given by the irreducible components of $I$.
\end{abstract}

\section{Introduction} 
 {\it Strongly stable ideals} are monomial ideals defined by a simple condition,  and they appear as the generic initial ideals of homogeneous ideals in the characteristic 0 case (so they are also called  {\it Borel fixed ideals} in this case).

Extending an idea of \cite{NR}, the second author (\cite{Y12}) constructed the {\it alternative polarization} $\bpol(I)$ of a strongly stable ideal $I$. We briefly explain this notion here. Let $S=\kk[x_1, \ldots, x_n]$ be a polynomial ring over a field $\kk$.   For a monomial ideal $I$, $G(I)$ denotes the set of minimal monomial generators of $I$.  
If $I \subset S$ is a strongly stable ideal with  $\deg(\sm) \le d$ for all $\sm \in G(I)$, 
we consider a larger polynomial ring $\wS=\kk[ \, x_{i,j} \mid 1\le i \le n, 1 \le j \le d \, ]$ with the surjection
$f: \wS \ni x_{i,j} \longmapsto x_i \in S$. Then we can construct a squarefree monomial ideal $\bpol(I) \subset \wS$ 
(if there is no danger of confusion, we will simply write $\wI$ for $\bpol(I)$) satisfying the conditions $f(\wI)=I$ and $\beta_{i,j}^{\wS}(\wI) =\beta_{i,j}^S(I)$ for all $i,j$, where $\beta_{i,j}$  stands for the graded Betti number.  The alternative polarization is much more compatible with operations for strongly stable ideals than the standard polarization.  
% sheds new light on classical results on strongly stable ideals. 

On the other hand, the Alexander duality for squarefree monomial ideals  is a very powerful tool in the Stanley--Reisner ring theory. For a squarefree monomial ideal $I \subset S$, $I^\vee \subset S$ denotes its Alexander dual. 
There  is a one to one correspondence between the elements  of $G(I)$ and the irreducible components of $I^\vee$.  Let $\wS'=\kk[ \, y_{i,j} \mid 1\le i \le d, 1 \le j \le n \, ]$ be a polynomial ring with the isomorphism $(-)^\st:\wS \ni x_{i,j}\longmapsto y_{j,i} \in \wS'$.  For a strongly stable ideal $I$, there is a strongly stable ideal $I^* \subset \kk[y_1, \ldots, y_d]$ with $\bpol(I^*) = (\bpol(I)^\vee)^\st$. Clearly, the correspondence $I \longleftrightarrow I^*$ should be considered as the Alexander duality for  strongly stable ideals. 

After we finished an earlier version of this paper, we were informed that, in  Fl\o ystad \cite[\S6]{F},  the above duality has been constructed using the notion of generalized (co-)letterplace ideals.
Each approach has each advantage.  The paper \cite{F} treats the duality in a much wider context, but if one starts from the generator set $G(I)$, our construction is more direct (Proposition~\ref{gene} and Theorem~\ref{irred main} give a simple procedure to compute $I^*$ from $G(I)$).  
We will give a complete proof of the existence of the duality, since we will re-use ideas of the proof in later sections.

The outline of the paper is as follows. Section 2 is mainly  devoted to the proof of the existence of the dual $I^*$. If $I$ is a Cohen--Macaulay strongly stable ideal, $\wS/\wI$ is the Stanley--Reisner ring of a ball or a sphere (a ball in most cases), and its canonical module can be easily described. In Section~3, we show the formula 
$$
H(H_{\fm}^i (S/I),  \lambda^{-1})=\displaystyle \sum_{j\in \ZZ}\frac{\beta_{i-j,n-j}(I^{*}) \lambda^j}{(1-\lambda)^j}
$$
on the Hilbert series of the local cohomology module $H_\fm^i(S/I)$.  This is more or less a consequence of  a classical result \cite{ER}, and we will improve this formula later.   

In Section 4, we discuss the relation to the notion of a {\it squarefree strongly stable ideal}, which is a squarefree  analog of a strongly stable ideal. 
For a strongly stable ideal $I \subset S$, Aramova et al \cite{AHH} constructed a  squarefree strongly stable ideal $I^\sigma \subset T =\kk[x_1, \ldots, x_N]$ with $N \gg 0$.   
The class of squarefree strongly stable ideals  is closed under the (usual) Alexander duality of $T$, so our duality can be constructed through $I^\sigma$. However, without $\bpol(I)$, it is hard to compare the algebraic properties of $I^*$ with those of $I$.   

In Section 5, we give a procedure constructing the irreducible decomposition of $\bpol(I)$ from that of a strongly stable ideal $I$. 
As corollaries, we will give formulas on the arithmetic degree $\adeg(S/I)$ and $H(H_\fm^i(S/I), \lambda)$ from the irredundant irreducible decomposition 
$$I =\bigcap_{\ba \in E} \fm^\ba$$ 
with $E \subset  \Z+ \cup  (\Z+)^2 \cup \cdots \cup  (\Z+)^n$. 
Here, for $\ba = (a_1, \ldots, a_t) \in (\ZZ_{>0})^t$ with $t \le n$, 
$\fm^\ba$ denotes the irreducible ideal $(x_1^{a_1}, \ldots, x_t^{a_t})$ of $S$. In this situation, set 
$t(\ba):=t$, $e(\ba) :=a_t$, and  $w(\ba):=n- \sum_{i=1}^t a_i$. Then we have 
$$\adeg (S/I)=\sum_{\ba \in E}e(\ba)$$
and 
$$ H(H_{\fm}^i (S/I), \lambda^{-1})
= \left(\sum_{\substack{\ba \in E, \\ t(\ba)=n-i}} (\lambda^{w(\ba)} + \lambda^{w(\ba)+1}+ \cdots + \lambda^{w(\ba)+e(\ba)-1})\right)  /(1-\lambda)^i. $$

Section 6 gives additional results on the irreducible decompositions of strongly stable ideals. While a  strongly stable ideal $I$ is characterized by the  ``left shift property" on $G(I)$, Theorem~\ref{right shift} states that it is also characterized by  the ``right shift property" on the irreducible components of $I$.

\section{The construction of the duality} 
We introduce the convention and notation used throughout the paper. For a positive integer $n$, set $[n]:=\{1,\, \ldots ,\, n\}$. Let  $S:= \kk[x_1 , \ldots , x_n ]$ be a polynomial ring over a field $\kk$,  and $\fm = (x_1, \ldots, x_n)$ the unique graded maximal ideal of $S$. 
For a monomial ideal $I\subset S$, $G(I)$ denotes the set of minimal monomial generators of $I$. We say an ideal $I \subset S$ is {\it strongly stable}, if it is a monomial ideal, and the condition that $\sm \in G(I),\ x_i | \sm$ and $j<i$ imply $(x_j /x_i) \cdot \sm \in I$ is satisfied.

Let $d$ be a positive integer, and set
$$
\wS:= \kk[\, x_{i,j} \, | \, 1\le i \le n, 1 \le j \le d \,]. 
$$
Note that 
$$
\Theta := \{x_{i,1}-x_{i,j} \, |\, 1\leq i \leq n,\, 2\leq j \leq d\,\} \subset \wS 
$$
forms a regular sequence with the isomorphism $\wS /(\Theta ) \cong S$ induced by $\wS \ni x_{i, j} \longmapsto x_i \in S$. 
\begin{defn}\label{pol}
For a monomial ideal $I\subset S$, a {\it polarization} of $I$ is a squarefree monomial ideal $J\subset \wS$ satisfying the following conditions.
\begin{itemize}
\item[(1)] Through the isomorphism $\wS /(\Theta ) \cong S$, we have $\wS /(\Theta ) \otimes_{\wS} \wS /J \cong S/I$. 
\item[(2)] $\Theta$ forms a $\wS /J$-regular sequence.
\end{itemize}
\end{defn}
%Of course, the above definition is equivalent to the one given in Introduction. 

For $\ba =(a_1, \ldots, a_n)\in \NN^{n}$, $x^\ba$ denotes the monomial $\prod_{i=1}^{n} x_{i}^{a_i}\in S$. %For a monomial $\sm:=x^\ba$, set $\deg (\sm):=\sum_{i=1}^n a_i$.
For a monomial $x^\ba \in S$ with $\deg (x^\ba) \leq d$, set
$$
\pol (x^{\ba}) := \displaystyle \prod_{1\leq i \leq n} x_{i,1}x_{i,2}\cdots x_{i,a_i} \in \wS. 
$$
If $I\subset S$ is a monomial ideal with $\deg (\sm) \leq d$ for all $\sm\in G(I)$, then it is well-known that
$$
\pol (I) := ( \, \pol (\sm) \, | \, \sm \in G(I) \, )
$$
is a polarization of $I$, which is called the {\it standard polarization}.

\medskip
Any monomial $\sm \in S$ has a unique expression 
\begin{equation}\label{another expression}
\sm = \prod_{i=1}^{e} x_{\alpha_i}  \, \,  \, \, \, {\rm with} \, \, \, \,  1\leq \alpha_1 \leq \alpha_2 \leq \cdots \leq \alpha_e \leq n .
\end{equation}
If $e \, (=\deg (\sm)) \leq d$, we set
\begin{equation}\label{another expression2}
\bpol (\sm) := \prod_{i=1}^e x_{\alpha_i ,i} \in \wS.
\end{equation}
As another expression, for a monomial $x^\ba \in S$ with $\deg (x^\ba) \leq d$, set $b_i := \sum_{j=1}^i a_j$ for each $i \geq 1$ and $b_0 =0$.
Then
$$
 \bpol (x^\ba) = \prod_{\substack{1\leq i\leq n \\ b_{i-1}+1\leq j \leq b_i}}x_{i,j}\in \wS.
$$
For a monomial ideal $I \subset S$  with $\deg (\sm) \leq d$ for all $\sm \in G(I)$ (in the sequel, we always assume this condition), set
$$
\bpol (I) := (\, \bpol (\sm) \, | \, \sm \in G(I)\, ) \subset \wS. 
$$
See the beginning of Example~\ref{main example} below. 

In \cite{Y12}, the second author showed the following. 

\begin{thm}[{\cite[Theorem 3.4]{Y12}}]
If $I\subset S$ is a strongly stable ideal, then $\bpol (I)$ gives a polarization of $I$. 
\end{thm}

In the rest of the paper, the next fact is frequently used without comment.   

\begin{lem}
Let $I\subset S$ be a strongly stable ideal. 
For a monomial $\sm \in S$ with $\deg(\sm) \le d$, $\sm \in I$ if and only if $\bpol(\sm) \in \bpol(I)$. 
\end{lem}

\begin{proof}
The necessity is shown in \cite[Lemma~3.1]{Y12}, and the sufficiency is an easy exercise.
 \end{proof}

An {\it irreducible monomial ideal} of $S$ is an ideal of the form  $(\, x_i^{a_i} \, |\, a_i >0  )$ 
for some $\ba \in \NN^{n}$.
A presentation of a monomial ideal $I$ as an intersection $I=\bigcap_{i=1}^r Q_i$ of irreducible monomial ideals is called an {\it irreducible decomposition}. 
An intersection $I=\bigcap_{i=1}^r Q_i$ is {\it irredundant}, if none of the ideals $Q_i$ can be omitted in this presentation. Any monomial ideal has a unique irredundant irreducible decomposition $I=\bigcap_{i=1}^r Q_i$. In this case, each $Q_i$ is called an {\it irreducible component} of $I$. If $I$ is a squarefree monomial ideal, then the  irreducible components are nothing other than the associated primes.

If $I \subset S$ is a squarefree monomial ideal (equivalently, $S/I$ is the Stanley--Reisner ring of some simplicial complex), then the irreducible components of $I$ are of the form $\fm^F := (x_i \mid i \in F)$ for some $F \subset [n]$, and the ideal 
$$I^{\vee} := \Bigl( \, \prod_{i \in F} x_i \mid \text{$\fm^F$ is an irreducible component of $I$} \Bigr)$$ 
called the {\it Alexander dual} of $I$. Then we have $I^{\vee \vee} =I$. 
This duality is very important in the Stanley--Reisner ring theory. See, for example, \cite{ER,MS}. 

\begin{lem}\label{characterization}
For a monomial ideal $I \subset S$, the following conditions are equivalent.
\begin{itemize}
\item[(1)] $I$ is strongly stable.
\item[(2)] $\bpol(I) \subset \wS$ has an irreducible decomposition $\bigcap_{s=1}^r P_s$ satisfying the following property.  
\begin{itemize}
\item[($*$)]For each $s$, there is a positive integer $t_s$, and  integers 
$\gamma_i^{\<s\>}$ for $1 \le i \le t_s$ such that $P_s =( \, x_{i, \gamma^{\<s\>}_i} \mid 1 \le i \le t_s \, )$ and $1 \le \gamma_1^{\<s\>} \le  \gamma_2^{\<s\>} \le  \cdots \le \gamma_{t_s}^{\<s\>}$. 
\end{itemize}
\end{itemize}
\end{lem}

\begin{proof}
(1) $\Rightarrow$ (2): This is shown already in  \cite[Remark 3.3]{Y12}.

(2) $\Rightarrow$ (1): 
For a contradiction, assume that $\wI:=\bpol(I)$ satisfies the condition ($*$) but $I$ is not strongly stable. Then it is easy to see that there is some $\sm =x^\ba \in G(I)$ such that $x_{j+1} \, | \,  \sm$ and $(x_j/x_{j+1} ) \cdot \sm \not \in I$ for some $j < n$.  Then we have $\bpol((x_j/x_{j+1} ) \cdot \sm) \not \in \bpol(I)$, and it implies that $\bpol((x_j/x_{j+1} ) \cdot \sm) \not \in P_s =( x_{1,\gamma_1^{\<s\>}},  x_{2,\gamma_2^{\<s\>}},\ldots,  x_{t_s,\gamma_{t_s}^{\<s\>}})$ for some $s$.  As before, set $b_0 :=0$ and $b_i :=\sum_{j=1}^i a_j$ for $i \ge 1$. 
Since 
$$\bpol (\sm) = \prod_{\substack{1\leq i\leq n \\ b_{i-1}+1\leq j \leq b_i}}x_{i,j}, $$
we have 
$\gamma_i^{\< s\>} \not \in \{\, b_{i-1}+1, \ldots, b_i \, \}$ for all $i \ne j, j+1$, 
$\gamma_j^{\< s\>} \not \in \{\, b_{j-1}+1, \ldots, b_j +1\, \}$, and 
$\gamma_{j+1}^{\< s\>} \not \in \{\, b_j+2, \ldots, b_{j+1} \, \}$. 
Here we have $\bpol(\sm) \in \bpol (I) \subset P_s$, and it implies $\gamma_{j+1}^{\< s\>} = b_j+1$. 
Since $\gamma_j^{\<s\>} \le \gamma_{j+1}^{\< s\>} \, (= b_j+1)$ and $\gamma_j^{\< s\>} \not \in \{\, b_{j-1}+1, \ldots, b_j +1\, \}$, we have $\gamma_j^{\<s\>}  \le b_{j-1}$.  If $j \ge2$, combining $\gamma_{j-1}^{\<s\>} \le \gamma_j^{\<s\>}  \, 
(\le b_{j-1})$ with $\gamma_{j-1}^{\< s\>} \not \in \{\, b_{j-2}+1, \ldots, b_{j-1} \, \}$, we have $\gamma_{j-1}^{\< s\>} \le 
b_{j-2}$. Repeating this argument, we have $\gamma_1^{\< s\>} \le b_0$. 
Since $\gamma_1^{\< s\>} \ge 1$ and  $b_0 =0$, this is a contradiction. 
\end{proof}

Let $\wS':= \kk[\, y_{i,j} \, | \, 1\le i \le d, 1 \le j \le n \,]$ be a polynomial ring with the ring isomorphism 
$(-)^{\st}:\wS \to \wS'$ defined by $\wS \ni x_{i,j} \longmapsto y_{j, i} \in \wS'$.   

\begin{thm}[c.f. {\cite{F}}]
Let $I \subset S$ be a strongly stable ideal. Then there exists a strongly stable ideal $I^* \subset S':=\kk[y_1, \ldots, y_d]$ such
that  $\bpol(I^*) = (\bpol(I)^\vee)^\st$. 
\end{thm}

\begin{proof}
As before, set $\wI:=\bpol(I)$. 
There is a one to one correspondence between the irreducible components of $\wI$ and the elements of $G(\wI^\vee)$. 
If the irrdundant irreducible decomposition of $\wI$ is given by 
$$
\wI = \bigcap_{s=1}^r ( \, x_{i, \gamma^{\<s\>}_i} \mid 1 \le i \le t_s \, ) \subset \wS,
$$
then we have 
$$(\wI^\vee)^\st = \Bigl( \, \prod_{i=1}^{t_s} y_{\gamma^{\<s\>}_i, i} \mid 1 \le s \le r \, \Bigr) \subset \wS'. $$
Since $\gamma_1^{\<s\>} \le  \gamma_2^{\<s\>} \le  \cdots \le \gamma_{t_s}^{\<s\>}$ by Lemma~\ref{characterization}, we have $\bpol(I^*)= (\wI^\vee)^\st$ for 
$$
I^* = \Bigl( \, \prod_{i=1}^{t_s} y_{\gamma^{\<s\>}_i } \mid 1 \le s \le r \, \Bigr)  \subset S'. 
$$

There also exists  a one to one correspondence between the irreducible components of $\wI^\vee$ and the elements of $G(\wI)$, equivalently, the elements of $G(I)$. If the monomial $\sm$ in  \eqref{another expression} belongs to $G(I)$, the irreducible component of $\wI^\vee$ given by $\sm$ is of the form  $(x_{\alpha_1,1}, x_{\alpha_2,2}, \ldots, x_{\alpha_e,e})$ by the expression \eqref{another expression2}. 
Then the  corresponding  irreducible component of $(\wI^\vee)^\st \, (= \bpol(I^*))$ is $(y_{1, \alpha_1}, \ldots, y_{e, \alpha_e}) \subset \wS'$. Since $\alpha_1 \le \cdots \le \alpha_e$, $I^*$ is strongly stable 
by Lemma~\ref{characterization}. 
\end{proof}
 
The above theorem gives a duality between strongly stable ideals $I \subset S=\kk[x_1, \ldots, x_n]$ whose generators have degree at most $d$ and strongly stable ideals $I^* \subset S'=\kk[y_1, \ldots, y_d]$ whose generators have degree at most $n$.

\begin{ex}\label{main example}
For a strongly stable ideal $I = (x_{1}^2, x_{1}x_{2}, x_{1}x_{3}, x_{2}^2 , x_{2}x_{3})$, we have
\begin{eqnarray*}
\bpol (I) &=& (\, x_{1,1}x_{1,2},\,  x_{1,1}x_{2,2},\,  x_{1,1}x_{3,2},\,  x_{2,1}x_{2,2},\,  x_{2,1}x_{3,2}\, ) \\
             &=& (\, x_{1,1},\,  x_{2,1}\, )\cap (\, x_{1,1},\,  x_{2,2},\,  x_{3,2}\, ) \cap (\, x_{1,2},\,  x_{2,2},\,  x_{3,2}\, ) \\
 \bpol (I)^\vee &=& (\, x_{1,1}x_{2,1},\,  x_{1,1}x_{2,2}x_{3,2},\,  x_{1,2}x_{2,2}x_{3,2}\, ) \\
 {(\bpol (I)^\vee)}^\st &=& (\, y_{1,1}y_{1,2},\,  y_{1,1}y_{2,2}y_{2,3},\,  y_{2,1}y_{2,2}y_{2,3}\, ), 
\end{eqnarray*}
hence the dual strongly stable ideal is given by 
$$
I^* = (\, y_{1}^2, \,  y_{1}y_{2}^2 ,\,   y_{2}^3 \, ). 
$$

On the other hand, if we use the standard polarization, we have 
\begin{eqnarray*}
\pol (I) &=&(\, x_{1,1}x_{1,2},\,  x_{1,1}x_{2,1},\,  x_{1,1}x_{3,1},\,  x_{2,1}x_{2,2},\, x_{2,1}x_{3,1}\, ) \\
           &=&(\, x_{1,1}, x_{2,1}\, )\cap (\, x_{1,1}, x_{2,2}, x_{3,1}\, )\cap (\, x_{1,2}, x_{2,1}, x_{3,1}\, ) \\
\pol (I)^\vee &=&(\, x_{1,1}x_{2,1},\,  x_{1,1}x_{2,2}x_{3,1},\,  x_{1,2}x_{2,1}x_{3,1}\, ).  
\end{eqnarray*}
Here $(\pol (I)^\vee)^\st =(y_{1,1}y_{1,2},\,  y_{1,1}y_{1,3}y_{2,2},\,  y_{1,2}y_{1,3}y_{2,1} )$  can not be the standard or altarnative polarization of any ideal. 
\end{ex}

\medskip

The next two results are implicitly contained in Fl\o ystad  \cite{F}.  However  they are stated in the context of the preceding papers \cite{FGH, DFN}, where the words ``letterplace ideal" and ``coletterplace ideals" are used in the narrow sense (see Remark~\ref{Gunnar} below).  

\begin{prop}\label{letterpalce}
If $I \subset S$ is a strongly stable ideal with $\sqrt{I}= \fm$, then $\bpol(I)$ (more precisely, $\bpol(I)^\st$) is the letterplace ideal $L(\cJ; d, [n])$ in the sense of \cite{DFN}. Here $\cJ$ is an order ideal of $\Hom([n],[d])$. Conversely, any letterplace ideal  $L(\cJ; d, [n])$ arises in this way from a strongly stable ideal $I$ with $\sqrt{I}= \fm$. 
\end{prop}

\begin{proof}
If $I \subset S$ is a strongly stable ideal with $\sqrt{I}= \fm$, then the dual $I^* \subset S'=\kk[y_1, \ldots, y_d]$ is a strongly stable ideal whose minimal generators all have degree $n$.  
As shown in \cite[\S 6.1]{FGH}, $\bpol(I^*)$ is a co-letterplace ideal $L([n],d;  \cJ)$ for some order ideal $\cJ \subset \Hom([n], [d])$. Then the Alexander dual  of $\bpol(I^*)$, which coincides with $\bpol(I)^\st$,  is the letterplace  ideal  $L(\cJ; d, [n])$ by definition. 

The second assertion follows from the fact that any co-letterplace ideal  $L([n],d; \cJ)$ is the $\bpol(-)$ of some strongly stable ideal whose generators all have degree $n$.  %(i.e.,  the converse of the fact used above).  
\end{proof}

\begin{rem}\label{Gunnar}
In \cite{F}, Fl\o ystad generalized the notions of a (co-)letterplace ideal so that $\bpol(I)$ of any strongly 
stable ideal $I$ belongs to these classes (one of the crucial points is considering an order ideal $\cJ$  in $\Hom([n], \NN)$, not in $\Hom([n], [d])$). Through this idea, he gave the duality.
\end{rem}

For a monomial $x^\ba \in S$ with $\ba=(a_1, \ldots, a_n) \in \NN^n$, set $\nu(x^\ba) := \max \{ \, i \mid  a_i > 0 \,  \}$. It is well-known that if $I$ is strongly stable, then 
$$\pd_S (S/I) = \max \{ \, \nu(\sm) \mid \sm \in G(I) \} \quad \text{and} \quad 
\height(I) =\max\{ \, i \mid x_i \in \sqrt{I} \, \}.$$
Hence, for a strongly stable ideal $I$ with $\height(I)=c$, $S/I$ is Cohen--Macaulay if and only if $\nu(\sm) \le c$ for all $\sm \in G(I)$, if and only if $\sm \in \kk[x_1, \ldots, x_c]$ for all $\sm \in G(I)$.  Of course, $\wS/\bpol(I)$ is Cohen--Macaulay if and only if so is $S/I$. 

\begin{cor} Let $(0) \ne I \subset S$ be a Cohen--Macaulay strongly stable ideal, and set $\wI:=\bpol(I)$.  
Then $\wS/\wI$ is the Stanley-Reisner ring of a ball or a sphere. More precisely, if $n \ge 2$, then $\wS/\wI$ is the Stanley-Reisner ring of a ball. 
\end{cor}

If $n=1$, then $I=(x^e)$ for some $e \le d$. Hence $\wI= (x_{1,1}x_{1,2} \cdots x_{1,e})$, and 
$\wS/\wI$ is the Stanley-Reisner ring of a sphere (resp. ball) if $e=d$ (resp. $e < d$).  

\begin{proof}
First, assume that $\sqrt{I} =\fm$. In this case, $\wI$ is a letterplace ideal $L(\cJ; d, [n])$ by Proposition~\ref{letterpalce}, and the assertion follows from \cite[Theorem~5.1]{DFN} (note that the poset $[n]$ is an antichain if and only if $n=1$). 

If  $\sqrt{I} \ne \fm$ (equivalently, $c:= \height(I) < n$),  then we have $I=JS$ for a  strongly stable ideal $J \subset \kk[x_1, \ldots, x_c]$ with $\sqrt{J} = (x_1, \ldots, x_c)$. Moreover, the simplicial complex associated with $\wI$ is the cone over the one associated with $\bpol(J)$. So the assertion can be reduced to the first case.  
\end{proof}

 For $x^\ba \in S$ with $\deg (x^\ba) \leq d$ and $l := \nu(x^\ba)$,  set 
$$\mu(x^\ba) := \Bigl(\prod_{i=1}^{l-1} x_{i,b_i+1} \Bigr) \cdot \bpol(x^\ba),$$
where $b_i := \sum_{j=1}^i a_j$ for each $i$ as before. 
In \cite{OY}, R. Okazaki and the second author constructed a minimal $\wS$-free resolution $\wP_\bullet$ of $\bpol(I)$ of a strongly stable ideal $I$. If $S/I$ is a Cohen-Macaulay ring of codimension $c$, 
the ``last" term $\wP_c$ of the minimal free resolution is isomorphic to 
$$\bigoplus_{\substack{\sm \in G(I) \\ \nu(\sm) =c}} \wS(-\deg(\mu(\sm))).$$

We also set 
$$\wX:=\prod_{\substack{ 1 \le i \le n \\ 1\le j \le d} }x_{i,j}$$
and 
$$\omega(\sm) := \wX /\mu(\sm)$$
for $\sm \in G(I)$. 

\begin{cor}\label{canonical}
Let $(0) \ne I \subset S$ be a Cohen--Macaulay strongly stable ideal with $\height(I)=c$, and set $\wI:=\bpol(I)$. 
Then the canonical module $\omega_{\wS/\wI}$  is isomorphic to the ideal of $\wS/\wI$ generated by (the image of) $\{ \, \omega(\sm) \mid \sm \in G(I), \nu(\sm)=c \, \}.$
\end{cor}

\begin{proof}
Since $\wS/\wI$ is the Stanley-Reisner ring of a ball or a sphere, the  canonical module $\omega_{\wS/\wI}$ is isomorphic to a multigraded  ideal of  $\wS/\wI$. 
Since $\omega_{\wS/\wI} =\Ext_{\wS}^c(\wS/\wI, \omega_{\wS})$ and $\omega_{\wS}$ is isomorphic to the principal ideal $(\wX)$ of $\wS$, $\omega_{\wS/\wI}$ is 
a quotient of 
$$\Hom_{\wS}(\wP_c, \omega_{\wS}) \cong \bigoplus_{\substack{\sm \in G(I) \\ \nu(\sm) =c}} \wS(-\deg(\omega(\sm))).$$
So we are done. 
\end{proof}

For a Cohen--Macaulay strongly stable ideal $I$, the canonical module $\omega_{S/I}$ of $S/I$ itself is isomorphic to $\omega_{\wS/\wI} \otimes_{\wS} \wS/(\Theta)$ and $\Theta$ forms a $(\omega_{\wS/\wI})$-regular sequence, where $
\Theta = \{x_{i,1}-x_{i,j} \, |\, 1\leq i \leq n,\, 2\leq j \leq d\,\}$. 
However, $\omega_{S/I}$ is  not isomorphic to an ideal of $S/I$ in general. 
 
We also remark that \cite[Corollary~4.3]{DFN} gives a description of the canonical module of the quotient ring of a letterplace ideal, and it also works in the case of  Corollary~\ref{canonical}. 
However, our description  is much  simpler in this case.

\section{The Hilbert series of $H_{\fm}^i(S/I)$}
Let $R=\kk[x_1, \ldots, x_m]$ be a polynomial ring. 
For a $\ZZ$-graded $R$-module $M$,  $H(M,\lambda)$ denotes the Hilbert series $\sum_{i \in \ZZ} (\dim_\kk M_i) \lambda^i$ of $M$.  
Let  $\omega_R$  denote the graded canonical module $R(-m)$ of $R$.  
%For example, $\omega_S \cong S(-n)$ and  $\omega_{\wS} \cong \wS(-nd)$.

The  following must be a fundamental formula on the Alexander duality of Stanley-Reisner ring theory, but we cannot find any reference.    

\begin{lem}\label{ext-betti}
Let $R=\kk[x_1, \ldots, x_m]$ be a polynomial ring, and $I \subset R$ a squarefree monomial ideal. 
Then  we have 
$$
H(\Ext_R^{m-i}(R/I, \omega_R), \lambda) = \displaystyle \sum_{j\ge 0} \frac{\beta_{i-j,m-j}(I^\vee)\lambda^j}{(1-\lambda)^j}. 
$$
Here $I^\vee \subset R$ is the Alexander dual of $I$, and 
$\beta_{p,q}(I^\vee)$ is the graded Betti number of $I^\vee$, that is, the dimension of $[\Tor_p^{R}(I^\vee, \kk)]_q$. 
\end{lem}

\begin{proof}
For $\ba =(a_1, \ldots, a_m)\in \NN^{m}$, the vector  $\bba =(\overline{a}_1, \ldots, \overline{a}_m)\in \NN^{m}$ is defined by 
$$
\overline{a}_i =
\begin{cases}
 1 & \text{ if $a_i \ge 1$,}  \\
 0 & \text{if $a_i=0$.}
 \end{cases}
$$

By \cite[Theorem 2.6]{Y00}, $\Ext_R^i (R/I, \omega_R)$ is a squarefree module. Hence we have $[\Ext_R^i (R/I, \omega_R)]_\ba= 0$ for all $\ba \in \ZZ^m \setminus   \NN^m$, and  
$$
[\Ext_R^i (R/I, \omega_R)]_\ba  \cong  [\Ext_R^i (R/I, \omega_R)]_{\bba} 
$$
for all $\ba \in \NN^m$.  Furthermore, it is well-known (cf., \cite[Theorem~3.4]{Y00}) that  
$$
[\Ext_R^i (R/I, \omega_R)]_{\bba} 
\cong  
[\Tor_{m- |\bba | -i}^R(\wI^{\vee},\kk)]_{\one-\bba}.   
$$
Here we set $\one :=(1, \ldots, 1 ) \in \NN^m$, and $|\bb| := \sum_{i=1}^m b_i$ for $\bb =(b_1, \ldots, b_m)\in \NN^{m}$. 
It is also well-known that  $[\Tor_{i}^R(\wI^{\vee},\kk)]_\ba \ne 0$ for $\ba \in \ZZ^m$ implies $\ba$ is a 0-1 vector.   %See, for example, \cite{Y00}. 

So we have
$$\dim_{\kk} [\Ext_R^{m-i}(R/I, \omega_R)]_0=  \beta_{i,m}(I^\vee)$$
and 
\begin{eqnarray*}
\dim_{\kk} [\Ext_R^{m-i}(R/I, \omega_R)]_l 
&=&  \displaystyle \sum_{j=1}^{l} \displaystyle \sum_{\substack{\ba \in \NN^m \\ |\ba|=l, |\bba|=j}} \dim_{\kk} [\Ext_R^{m-i}(R/I, \omega_R)]_\ba \\
&=&  \displaystyle \sum_{j=1}^{l} \displaystyle \sum_{\substack{\ba \in\NN^m \\ \ba =\bba, |\ba|=j}} \binom{l-1}{l-j}\dim_{\kk} [\Ext_R^{m-i}(R/I, \omega_R)]_{\ba} \\
&=& \displaystyle \sum_{j=1}^{l} \displaystyle \sum_{\substack{\ba \in\NN^m \\ \ba =\bba, |\ba|=j}} \binom{l-1}{l-j}\dim_{\kk} [\Tor_{i-j}^R (I^{\vee},\kk)]_{\one - \ba} \\
&=& \displaystyle \sum_{j=1}^{l} \binom{l-1}{l-j} \beta_{i-j,m-j}(I^\vee)\\
\end{eqnarray*}
for $l > 0$.  So the assertion follows from the following computation  
\begin{eqnarray*}
\sum_{j\ge 0} \frac{\beta_{i-j,m-j}(I^\vee) \lambda^j}{(1-\lambda)^j} 
&=&\beta_{i,m}(I^\vee)+\sum_{j \ge 1} \left\{  \beta_{i-j, m-j}(I^\vee) \lambda^j \cdot \sum_{p\ge0} \binom{j+p-1}{p}\lambda^p \right\} \\
%&=& \beta_{i, m}(I^\vee)+ \sum_{j \ge 1} \left\{  \beta_{i-j, m-j}(I^\vee) \cdot \sum_{l\ge j}\binom{l-1}{l-j} \lambda^l \right\}\\
&=& \beta_{i, m}(I^\vee)+ \sum_{l\ge1} \left\{ \sum_{j=1}^{l} \binom{l-1}{l-j}\beta_{i-j, m-j}(I^\vee)       \right\}  \lambda^l \\
&=& \dim_{\kk} [\Ext_R^{m-i}(R/I, \omega_R)]_0 + \sum_{l\ge1} \dim_{\kk} [\Ext_R^{m-i}(R/I, \omega_R)]_l   \cdot \lambda^l, 
\end{eqnarray*}
where $l:=j+p$. 
\end{proof}

\begin{cor}\label{extbpol}
For  a strongly stable ideal   $I\subset S$ with $\wI := \bpol (I)$, we have 
$$
H(\Ext_{\wS}^{nd-i}(\wS/\wI, \omega_{\wS}), \lambda) = \displaystyle \sum_{j\ge 0} \frac{\beta_{i-j,nd-j}(I^*)\lambda^j}{(1-\lambda)^j}. 
$$
%Here $\beta_{p,q}(I^{*})$ is the graded Betti number of $I^{*}$, that is, the dimension of $[\Tor_p^{S'}(I^*, \kk)]_q$. 
\end{cor}

\begin{proof}
The assertion follows from Lemma~\ref{ext-betti} (applying to $\wI \subset \wS$) and 
the equality $\beta_{p,q}(\wI^\vee) = \beta_{p,q}(I^*)$. 
\end{proof}

\begin{thm}\label{localcoh}
Let $I\subset S$ be a strongly stable ideal. Then the Hilbert series of the local cohomology module $H_\fm^i(S/I)$ can be described as follows.
$$
H(H_{\fm}^i (S/I),  \lambda^{-1})=\displaystyle \sum_{j\in \ZZ}\frac{\beta_{i-j,n-j}(I^{*}) \lambda^j}{(1-\lambda)^j}. 
$$
%for all $i\in \ZZ$. 
\end{thm}

\begin{proof}
Set $\Theta := \{\, x_{i,1}-x_{i,j} \, | \, 1\leq i \leq n, \, 2\leq j \leq d \, \}$. By the full statement of  \cite[Theorem~3.4]{Y12}, if $\Ext_{\wS}^i (\wS/\wI, \wS) \ne 0$, then $\Theta$ forms an
$\Ext_{\wS}^i (\wS/\wI, \wS)$-regular sequence.  Hence we have
$$
[\wS/(\Theta)\otimes_{\wS} \Ext_{\wS}^{n-i} (\wS/\wI, \omega_{\wS})](nd-n) \cong \Ext_{S}^{n-i}(S/I, \omega_{S})
$$
and 
\begin{eqnarray*}
H(\Ext_{S}^{n-i}(S/I,\omega_S), \lambda)
&=& \lambda^{n-nd} \cdot H(\wS/(\Theta)\otimes_{\wS} \Ext_{\wS}^{n-i} (\wS/\wI, \omega_{\wS}), \lambda) \\
&=& \lambda^{n-nd}(1-\lambda)^{nd-n} \cdot H(\Ext_{\wS}^{n-i}(\wS/\wI , \omega_{\wS}),\lambda) \\
&=& \lambda^{n-nd}(1-\lambda)^{nd-n} \displaystyle \sum_{j\ge 0} \frac{\beta_{nd-n+i-j,nd-j}(I^*) \lambda^j}{(1-\lambda)^j}, 
\end{eqnarray*}
where the last equality follows from Corollary~\ref{extbpol}. 
Replacing $j$ by $nd-n+j$, we have 
\begin{eqnarray*}
H(H_{\fm}^i (S/I), \lambda^{-1})&=&H(\Ext_{S}^{n-i}(S/I, \omega_{S}),\lambda)\\ 
&=&  \lambda^{n-nd}(1-\lambda)^{nd-n} \displaystyle \sum_{j\ge n-nd} \frac{\beta_{i-j, n-j}(I^*)\lambda^{nd-n+j}}{(1-\lambda)^{nd-n+j}} \\
&=&  \sum _{j\ge n-nd} \frac{\beta_{i-j, n-j}(I^*) \lambda^{j}}{(1-\lambda)^j}.
\end{eqnarray*}
Here the first equality follows from the fact that $H_{\fm}^i (S/I)$ is the graded Matlis dual of $\Ext_{S}^{n-i}(S/I, \omega_{S})$.  
\end{proof}

\begin{cor}\label{components and lc}
Let $I\subset S$ be a strongly stable ideal. Then $S/I$ is a Cohen-Macaulay ring if and only if $I^*$ has a linear resolution.
\end{cor}

\begin{proof}
Follows from Theorem~\ref{localcoh}, or from \cite[Theorem~3]{ER}.   
%By Theorem~\ref{extbpol}, the Betti numbers $\beta_{i, i+j}(I^*)$ for $i=0,1, \ldots, n$ in the  $j$-linear part determine the Hilbert series of $H_\fm^{n-j}(S/I)$. So the assertion follows.  
\end{proof}

\begin{cor}
Let $I$ be a strongly stable ideal.  If the  irredundant irreducible decomposition of $\bpol(I)$ is of the form  
\begin{equation}\label{irred decomp of wI}
\bpol (I) = \bigcap_{s=1}^r (x_{i,\gamma_{i}^{\<s\>}} \, |\, 1\leq i \leq t_s \, ) \subset \wS, 
\end{equation}
then we have 
$$
H(H_{\fm}^i (S/I), \lambda^{-1})=\displaystyle \frac{\sum_{j\geq 1}\# \{ s\in [r]\, |\, t_s = n-i ,\, \gamma_{t_s}^{\<s\>}=j\, \} \lambda^{i-j+1}}{(1-\lambda)^i} .
$$
\end{cor}

\begin{proof}
By the additivity of the statement, it suffices to compute how an irreducible  component
$$
P_s = (x_{i,\gamma_{i}^{\<s\>}} \, |\, 1\leq i \leq t_s \, )
$$
of $\bpol(I)$ contributes to the Hilbert series $H(H_{\fm}^i (S/I), \lambda^{-1})$. 
For simplicity, set $\gamma = \gamma_{t_s}^{<s>}$ and $t=t_s$.
This component gives  
$$
\prod_{i=1}^t y_{\gamma_i^{<s>}} \in G(I^*).
$$
By the Eliahou-Kervaire formula (\cite{EK}), the contribution of $P_s$ to the Betti numbers of $I^*$ is 
\begin{align*}
\begin{cases}
 0 & \text{ if $j \neq  t$,}  \\
 \binom{\gamma -1}{i}  & \text{ if $j=t$,}
 \end{cases}
 \end{align*}
 for  $\beta_{i,i+j}(I^*)$, equivalently, 
 \begin{align*}
\begin{cases}
 0 & \text{ if $n-i \neq  t$,}  \\
 \binom{\gamma -1}{i-j}  & \text{ if $n-i=t$,}
 \end{cases}
 \end{align*}
 for $\beta_{i-j,n-j}(I^*)$. Hence, by Theorem \ref{localcoh}, $P_s$ concerns $H_{\fm}^i (S/I)$ if and only if $i=n-t$.
 Moreover, if $i=n-t$, the contribution to $H(H_{\fm}^i(S/I), \lambda^{-1})$ is the following 
 \begin{eqnarray*}
 \sum_{j=i-\gamma+1}^{i}\frac{\binom{\gamma -1}{i-j} \lambda^j}{(1-\lambda)^j} 
 &=&  \frac{\sum_{j=i-\gamma+1}^{i}(1-\lambda)^{i-j}\binom{\gamma -1}{i-j} \lambda^j}{(1-\lambda)^i}  \\
 &=&  \frac{\sum_{k=0}^{\gamma-1}(1-\lambda)^{k}\binom{\gamma -1}{k} \lambda^{i-k}}{(1-\lambda)^i}  \, \, \,   \, \, \, \, \, \, \, ({\rm here} \, \,  k=i-j)\\
  &=&  \frac{\bigl(\sum_{k=0}^{\gamma-1}(1-\lambda)^{k}\binom{\gamma -1}{k} \lambda^{\gamma-1-k}\bigr) \lambda^{i-\gamma+1}}{(1-\lambda)^i} \\
  &=& \frac{((1-\lambda)+\lambda)^{\gamma -1} \lambda^{i-\gamma +1}}{(1-\lambda)^i}\\
  &=& \frac{\lambda^{i-\gamma +1}}{(1-\lambda)^i}  .
 \end{eqnarray*}
 So the proof is completed.
\end{proof}

\begin{ex}
For the strongly stable ideal $I$ in Example~\ref{main example},  $\bpol(I)$ has two height 3 irreducible components $P_2 =  (x_{1,1},  x_{2,2},  x_{3,2})$ and $P_3= (x_{1,2},  x_{2,2},  x_{3,2} )$. Clearly, $\gamma_3^{\<2\>} = \gamma_3^{\<3\>} = 2$ in the above notation. Hence we have $H(H_\fm^0(S/I), \lambda^{-1}) =2\lambda^{-2+1} = 2 \lambda^{-1}$ by Corollary~\ref{components and lc}. 
\end{ex}

In Section 5, we will give a procedure constructing the irreducible decomposition  of $\bpol(I)$ from that of $I$ itself. After this, we will return to the Hilbert series of $H_{\fm}^i(S/I)$. See Corollary~\ref{components and lc2} below.

\section{Relation to squarefree strongly stable ideals}
We say an ideal $I \subset S$ is {\it squarefree strongly stable}, if it is a squarefree monomial ideal and the condition that $\sm \in G(I),\ x_i  \, | \, \sm$, $j<i$ and $x_j \! \not|  \, \sm$ imply $(x_j /x_i) \cdot \sm \in I$ is satisfied.
For our study on (squarefree) strongly stable ideals,  the dimension of the ambient ring $S=\kk[x_1, \ldots, x_n]$ 
is not important.  So we consider the following equivalence relation. For monomial ideals $I \subset S_{(n)} :=\kk[x_1, \ldots, x_n]$ and $J \subset S_{(m)} :=\kk[x_1, \ldots, x_m]$, the relation $I \equiv J$ holds if the following condition is satisfied. 
\begin{itemize}
\item Without loss of generality, we may assume that $n \le m$. Then regarding $S_{(n)}$ as a subring of $S_{(m)}$ in the natural way, we have $G(I)=G(J)$.    
\end{itemize}

For a monomial $\sm \in S$ of the form \eqref{another expression},  set 
$$\sm^\sigma := \prod_{i=1}^e x_{\alpha_i+i-1} \in T,$$
where $T=\kk[x_1, \ldots, x_N]$ is a polynomial ring with $N \gg 0$.     
Aramova et al. \cite{AHH} showed that if $I \subset S$ is a strongly stable ideal then 
$$I^\sigma := (\, \sm^\sigma \mid \sm \in G(I) \, ) \subset T$$ 
is squarefree strongly stable. Conversely, any squarefree strongly stable ideal is of the form $I^\sigma$ for some strongly stable ideal $I$. 

Let $I \subset S$ be a strongly stable ideal, and $\wI:= \bpol(I) \subset \wS$ its alternative polarization. 
For $$\Theta_1:=\{ \, x_{i,j} -x_{i+1, j-1} \mid 1 \le i <n,   1< j \le d \, \},$$ we have an isomorphism $\wS/(\Theta_1) \cong T=\kk[x_1, \ldots, x_N]$ with $N=n+d-1$ induced by $\wS \ni x_{i,j} \longmapsto x_{i+j-1} \in T$. 
As shown in \cite[\S4]{Y12}, we have 
\begin{itemize}
\item[(1)] Through the isomorphism $\wS /(\Theta_1) \cong T$, we have $\wS /(\Theta_1) \otimes_{\wS} \wS /\wI \cong T/I^\sigma$. 
\item[(2)] $\Theta_1$ forms a $\wS /\wI$-regular sequence.
\end{itemize}

\begin{thm}\label{I^sigma}
Let $I$ be a strongly stable ideal. If the irredundant irreducible decomposition of $\bpol(I)$ is of the form 
\eqref{irred decomp of wI}, then we have 
$$I^\sigma =  \bigcap_{s=1}^r ( \, x_{\gamma^{\<s\>}_i +i-1} \mid 1 \le i \le t_s \, ) \subset T.$$
\end{thm}

\begin{proof}
As above, set $\wI :=\bpol(I)$. 
Since both $\wS/\wI$ and $T/I^\sigma$ are reduced, and 
$$\wS/ ( \, x_{i, \gamma^{\<s\>}_i} \mid 1 \le i \le t_s \, ) \otimes_{\wS} \wS/(\Theta_1) \cong T/ ( \, x_{\gamma^{\<s\>}_i +i-1} \mid 1 \le i \le t_s \, ),$$ it suffices to show that all associated primes of $T/I^\sigma \, (\cong  \wS/\wI  \otimes_{\wS} \wS/(\Theta_1) )$ come from those of $\wS/\wI$. 

%The set of height $c$ associated primes of $\wS/\wI$ is equal to the set of height $c$ associated primes of $\Ext_{\wS}^c(\wS/\wI, \omega_{\wS})$.  The same is true for $T/I^\sigma$. 
As shown in \cite[Theorem~3.2]{Y12}, $\wS/\wI$ is sequentially Cohen-Macaulay, that is, if $\Ext_{\wS}^c(\wS/\wI, \omega_{\wS})\ne 0$ then it is a Cohen-Macaulay module of codimension $c$. From \cite[pp.349--351]{V}, we see that  
$$(\text{the number of height $c$ associated primes of $\wI$}) =\deg  (\Ext_{\wS}^c(\wS/\wI, \wS))$$
and 
$$(\text{the number of height $c$ associated primes of $I^\sigma$}) =\deg  (\Ext_T^c(T/I^\sigma, T)).$$

By the same argument as the proof of \cite[Theorem 3.4]{Y12}, we can show that $\Theta_1$ forms an $\Ext_{\wS}^c(\wS/\wI, \wS)$-regular sequence (see also \cite[Proposition~4.1]{Y12}).  
Hence
$$\wS/ ( \Theta_1 ) \otimes_{\wS} \Ext_{\wS}^c(\wS/\wI, S)  \cong \Ext_T^c(T/I^\sigma, T),$$
and we have 
$$\deg  (\Ext_{\wS}^c(\wS/\wI, \wS)) = \deg  (\Ext_T^c(T/I^\sigma, T)).$$
So we are done. 
\end{proof}

\begin{cor}\label{I^sigma dual}
If $I$ is a strongly stable  ideal, we have $$(I^\sigma)^\vee \equiv (I^*)^\sigma,$$ where $\equiv$ is the relation defined above. 
\end{cor}

\begin{proof}
If the irredundant irreducible decomposition of $\bpol(I)$ is given as in \eqref{irred decomp of wI}, then both $(I^\sigma)^\vee$ and $ (I^*)^\sigma$ are equal to 
$$\Bigl( \, \prod_{i=1}^{t_s} x_{\gamma^{\<s\>}_i +i-1} \mid 1 \le s \le r \, \Bigr).$$
More precisely, $(I^*)^\sigma$ should be  an ideal with variables  $y_1, y_2\ldots$, but this is not essential. 
\end{proof}

The Alexander duals of squarefree strongly stable ideals already appeared in an earlier  paper \cite{HT} (of course, they knew that these are  squarefree strongly stable again).  %In this sense, our duality is not very new.  
However, the algebraic relation between $I$ and $I^\sigma$ is not clear, if one does not know $\bpol(I)$. 

\begin{ex}
Consider the strongly stable ideal $I = (x_{1}^2, x_{1}x_{2}, x_{1}x_{3}, x_{2}^2 , x_{2}x_{3})$  of Example~\ref{main example}. Then 
$$I^\sigma = (\, x_1x_2 , \, x_1x_3 , \, x_1x_4, \, x_2x_3, \, x_2x_4\, ) \\
              = (\, x_1,\, x_2\, )\cap  (\, x_1,\, x_3,\, x_4\, ) \cap (\, x_2,\, x_3,\, x_4\, )$$
and hence $(I^\sigma)^\vee  = ( x_1x_2,  x_1x_3x_4,  x_2x_3x_4 ).$

On the other hand, since $I^* = (y_{1}^2,  y_{1}y_{2}^2 ,   y_{2}^3  )$, we have $(I^*)^\sigma=( y_1y_2, y_1y_3y_4, y_2y_3y_4)$. 
\end{ex}

\begin{comment}
\begin{rem}
S. Murai \cite{M} has generalized the squarefree  operation $(-)^\sigma$ in the following way.  let $\{a_i\}_{i \in \NN}$ be a strictly increasing integer  sequence  with $a_0=0$.   For s monomial $\sm \in S$ of the form \eqref{another expression},  set 
$$\sm^\sigma :=  \prod_{i=1}^e x_{\alpha_i +a_{i-1}} \in T,$$
where $T=\kk[x_1, \ldots, x_N]$ is a polynomial ring with $N \gg 0$.   
For a strongly stable ideal $I \subset S$, he also set 
$$I^{\sigma(a)}:= ( \, \sm^{\sigma(a)} \mid \sm \in G(I) \, ) \subset T,$$
which is always squarefree. 
If $a_i=i$ for all $i$, then $(-)^{\sigma(a)}$ coincides  
with the operation $(-)^{\sigma}$. 

Let $L_a$ be the linear subspace of $S_1$ spanned by 
$$X_a := \{ \, x_{i,j} -x_{i', j'} \mid i+a_{j-1}=i'+a_{j'-1}  \, \},$$
and take a subset $\Theta_a \subset X_a$ so that it forms a basis of $L_a$. 
Clearly,  $\Theta_a$ is a $\wS$-regular sequence, and the ring homomorphism 
$\psi_a : \wS \to T \, (=\kk[x_1, \ldots, x_N])$ 
defined by $\wS \ni x_{i,j} \longmapsto x_{i+a_{j-1}} \in T$ induces the isomorphism 
$\wS/(\Theta_a) \cong T$.    
Then $\Theta_a$ forms a $\wS/\wI$-regular sequence, and we have 
$$\wS/(\Theta_a)\otimes_{\wS} \wS/\wI  \cong T/I^{\sigma(a)}.$$ 
\end{rem}
\end{comment}

\section{The irreducible components of $I$ and $\bpol(I)$} 
\begin{defn}
For $\ba = (a_1,\ldots ,a_t)\in (\Z+)^t$, set 
$$\Psi(\ba) := 
\left\{ (b_1, \ldots ,b_{t-1},c) \in (\Z+)^t \left|
 \begin{array}{l}
 b_i = \Bigl ( \displaystyle \sum_{j=1}^i a_j \Bigr) -i+1 \ \ \text{for} \ \  i < t, \\
 b_{t-1} \le c \le b_{t-1} + a_t-1  
 \end{array}
 \right.\right\}.
$$
Here, if $t=1$, then we set $1 \le c \le a_1$. 
\end{defn}

\begin{rem}
In the above situation, we have $|\Psi(\ba)| = a_t$. Moreover, 
for $\bb = (b_1,,\ldots ,b_{t-1},c)\in \Psi(\ba)$, we have $1 \le b_1 \le \cdots \le  b_{t-1} \le c$. 
\end{rem}

\begin{ex}
If $\ba=(3,2,1,2)$, then $\Psi(\ba)= \{ \, (3,4,4,4), (3,4,4,5) \, \}$. 
\end{ex}

For $\ba =(a_1, \ldots, a_t) \in (\Z+)^t$ with $t \le n$, set $\fm^\ba := (x_1^{a_1}, \ldots, x_t^{a_t}) \subset S$. 
If $(0) \ne I \subset S$ is a strongly stable ideal, then an irreducible component of  $I$ is of the form $\fm^\ba$ for some $\ba \in (\Z+)^t$. 
Hence there is some 
$$E \subset  \Z+ \cup  (\Z+)^2 \cup \cdots \cup  (\Z+)^n$$
such that 
\begin{equation}\label{irred decomp}
I = \bigcap_{\ba \in E} \fm^\ba
\end{equation}
is the irredundant irreducible decomposition.  

For $\bb = (b_1,\ldots ,b_t)\in \Psi(\ba)$, we set 
$$\wfm^\bb:= (x_{1,b_1}, x_{2, b_2}, \ldots, x_{t, b_t}) \subset \wS.$$

\begin{thm}\label{irred main}
Let $I$ be  a strongly stable ideal whose  irredundant irreducible decomposition is given by \eqref{irred decomp}. 
Set $\Psi(E):= \bigcup_{\ba \in E} \Psi(\ba)$. Then 
\begin{equation}\label{irred main eq}
\bpol(I) = \bigcap_{\bb \in \Psi(E)} \wfm^\bb
\end{equation}
is the irredundant irreducible decomposition.  
\end{thm}

It is easy to see that $\Psi(E)= \bigsqcup_{\ba \in E} \Psi(\ba)$. We will implicitly  use this fact in the arguments below. 

To prove the theorem, we need some preparation. 
Let $I$ be a strongly stable ideal whose  irredundant irreducible decomposition is given by \eqref{irred decomp}. We decompose $E$ into three parts 
$E_0 =\{ \, (a_1, \ldots, a_t) \in E \mid t < n \, \}$, $E_1 =\{ \, (a_1, \ldots, a_n ) \in E \mid a_n =1 \, \} $ and   $E_2 =\{ \, (a_1, \ldots, a_n ) \in E \mid a_n \ge 2  \, \}$. 

\begin{lem}\label{colon}
With the above notation, $I:x_n$ is a strongly stable ideal (not necessarily minimally) generated by 
$$\{ \,  \sm \in G(I) \mid \text{$x_n$ does not divide $\sm$} \, \} \cup 
\{ \, \, \sm/x_n \mid  \sm \in G(I), \text{$x_n$ divides $\sm$} \, \}.$$
Moreover, its irredundant irreducible decomposition is given by  
\begin{equation}\label{I:x_n}
I:x_n = \Bigl( \bigcap_{\ba \in E_0} \fm^\ba\Bigr) \cap \Bigl( \bigcap_{\ba \in E_2} \fm^{\ba-\be_n}\Bigr),
\end{equation}
where $\be_n$ is the $n$-th unit vector $(0,0, \ldots,1 ) \in \ZZ^n$. 
\end{lem}

\begin{proof}
The first and second assertions are clear. 
To see the last assertion, note that 
$$I:x_n = \Bigl ( \bigcap_{\ba \in E} \fm^\ba \Bigr ) :x_n = \bigcap_{\ba \in E} (\fm^\ba :x_n), $$
and
$$
\fm^\ba :x_n = \begin{cases}
\fm^\ba & \text{if $\ba \in E_0$,} \\
S & \text{if $\ba \in E_1$,} \\
\fm^{\ba -\be_n} & \text{if $\ba \in E_2$.}
\end{cases}
$$
So \eqref{I:x_n} holds.  Since there is no inclusion among $\fm^\ba$ for $\ba \in E_0$ and $\fm^{\ba-\be_n}$ for $\ba \in E_2$, the decomposition  \eqref{I:x_n} is irredundant. 
\end{proof}

Set 
$$\oI:= ( \, \sm \in G(I) \mid \text{$x_n$ does not divide $\sm$} \, ).$$
For $\ba =(a_1, \ldots, a_t) \in E$, set 
\begin{equation}\label{varphi}
\varphi(\ba) = \begin{cases}
\ba & \text{if $t <n$,} \\
(a_1, \ldots, a_{n-1})  & \text{if $t =n$.} 
\end{cases}
\end{equation}

\begin{lem}\label{oI}
With the above notation, we have the following. 
\begin{itemize}
\item[(1)] $\oI$ is a strongly stable ideal, and 
$$\oI = \bigcap_{\ba \in E} \fm^{\varphi(\ba)}$$
is a (possibly redundant)   irreducible decomposition.
\item[(2)] For $\ba \in E_1$, $\fm^{\varphi(\ba)}$ is an  irreducible component of $\oI$. 
\end{itemize}
\end{lem}

\begin{proof}
(1): Easy. 

(2) For a contradiction, assume that $\fm^{\varphi(\ba)}$ for $\ba \in E_1$ is not an irreducible component.  Then there is some $\ba' \in E \setminus \{ \ba \}$ such that $\fm^{\varphi(\ba')} \subset \fm^{\varphi(\ba)}$. Since $\ba \in E_1$, we have $\fm^{\ba'} \subset \fm^{\ba}$, and this is a contradiction. 
\end{proof}

Next we will study how to recover  a strongly stable ideal $I$ from $I:x_n$ and $\oI$. 
Let 
\begin{equation}\label{decompositions}
I:x_n = \bigcap_{\ba \in F} \fm^\ba \qquad \text{and} \qquad \oI  = \bigcap_{\ba \in G} \fm^\ba 
\end{equation}
be the irredundant irreducible decompositions. 
Decompose $F$ into 
$$F_0 =\{ \, (a_1, \ldots, a_t ) \in F \mid t < n \, \} \qquad  \text{and}  \qquad F_1:= (F \setminus F_0) \subset  (\Z+)^n,$$ 
and set $\varphi(F) :=\{ \, \varphi(\ba) \mid \ba \in F \, \}$, where $\varphi$ is the function defined in \eqref{varphi}. By  Lemmas~\ref{colon} and \ref{oI}, if $\ba \in G \setminus \varphi(F)$, then $\ba$ is of the from $(a_1, \ldots, a_{n-1})$ and $\fm^{\ba \oplus \be_n}$ is an irreducible component of $I$, where we set $\ba \oplus \be_n := (a_1, \ldots, a_{n-1},1)$. 

\begin{lem}\label{recover} 
With the above notation, we have the irredundant irreducible decomposition
$$I= \Bigl(\bigcap_{\ba \in F_0} \fm^\ba\Bigr) \cap \Bigl(\bigcap_{\ba \in F_1} \fm^{\ba+\be_n}\Bigr) \cap 
\Bigl(\bigcap_{\ba \in G \setminus \varphi(F)} \fm^{\ba \oplus \be_n}\Bigr).$$
\end{lem}

\begin{proof}
Easily follows from Lemmas~\ref{colon} and \ref{oI}.
\end{proof}

\noindent{\it The proof of Theorem~\ref{irred main}.} 
We prove the theorem by double induction on $n$ and 
$$d(I) := \sum_{\sm \in G(I)} \deg (\sm).$$  

Let $I$ be a strongly stable ideal. We may assume that $x_n$ divides some $\sm \in G(I)$. 
In fact, if this  is not the case, we can replace $I$ by $I \cap \kk[x_1, \ldots, x_{n-1}]$, and the induction works.  
Under this assumption, both $d(I:x_n)$ and $d(\oI)$ are smaller than $d(I)$. 
By the induction hypothesis, if $I:x_n$ and $\oI$ have irreducible decompositions of the form \eqref{decompositions}, we have irreducible decompositions 
$$ \bpol(I:x_n)= \bigcap_{\bb \in \Psi(F)} \wfm^\bb  \qquad \text{and} \qquad 
\bpol(\oI)= \bigcap_{\bb \in \Psi(G)} \wfm^\bb.$$
In the sequel, for $\ba=(a_1, \ldots, a_n) \in (\Z+)^n$, consider  the vector $(b_1, \ldots, b_n)$ with $b_i = \Bigl ( \displaystyle \sum_{j=1}^i a_j \Bigr) -i+1$ for $i=1, \ldots, n$. In this case, 
\begin{equation}\label{Psi(a)}
\Psi(\ba) =\{ \,  (b_1, \ldots, b_{n-1}, c) \mid b_{n-1} \le c \le b_n \, \} 
\end{equation}
and 
$$
\Psi(\ba+\be_n) = \Psi(\ba) \cup \{  \, (b_1, \ldots,  b_{n-1}, b_n +1)\, \}.
$$
Set $\wba := (b_1, \ldots,  b_{n-1}, b_n +1)$.  

For $\ba=(a_1, \ldots, a_{n-1}) \in (\Z+)^{n-1}$, we have 
$$\Psi(\ba \oplus \be_n)= \{ \, (b_1, \ldots,  b_{n-1}, b_{n-1}) \, \},$$ 
where $b_i = \Bigl ( \displaystyle \sum_{j=1}^i a_j \Bigr) -i+1$ for $i=1, \ldots, n-1$. 
Set $\oba:=  (b_1, \ldots,  b_{n-1}, b_{n-1}) $.

By Lemma~\ref{recover}, it is enough to show  
$$
\bpol(I)  =  \Bigl(\bigcap_{\bb \in \Psi(F_0)} \wfm^\bb\Bigr) \cap \Bigl(\bigcap_{\substack{\ba \in F_1 \\ \bb \in 
\Psi(\ba + \be_n)}} \wfm^\bb \Bigr) \cap 
\Bigl(\bigcap_{\ba \in G \setminus \varphi(F)} \wfm^{\oba} \Bigr). $$
Since 
\begin{eqnarray*}
(\text{the right hand side}) &=& \Bigl(\bigcap_{\bb \in \Psi(F)} \wfm^\bb\Bigr) \cap \Bigl(\bigcap_{\ba \in F_1} \wfm^{\wba} \Bigr ) \cap  \Bigl(\bigcap_{\ba \in G \setminus \varphi(F)} \wfm^{\oba} \Bigr)  \\
&=& \bpol(I:x_n) \cap \Bigl(\bigcap_{\ba \in F_1} \wfm^{\wba} \Bigr )   \cap  \Bigl(\bigcap_{\ba \in G \setminus \varphi(F)} \wfm^{\oba} \Bigr),
\end{eqnarray*}
it suffices to show that 
\begin{equation}\label{ETS}
\bpol(I)  =\bpol(I:x_n) \cap \Bigl(\bigcap_{\ba \in F_1} \wfm^{\wba} \Bigr )   \cap  \Bigl(\bigcap_{\ba \in G \setminus \varphi(F)} \wfm^{\oba} \Bigr). 
\end{equation}
First, we will prove the inclusion $\subset$ of \eqref{ETS}. Since  $\bpol(I)  \subset \bpol(I:x_n) $, it suffices to show that 
\begin{equation}\label{ETS2}
\bpol(\sm) \in \Bigl(\bigcap_{\ba \in F_1} \wfm^{\wba} \Bigr )   \cap  \Bigl(\bigcap_{\ba \in G \setminus \varphi(F)} \wfm^{\oba} \Bigr)
\end{equation}
for all $\sm  \in G(I)$. 

Take an arbitrary $\ba \in F_1$, and set $\wba= (b_1, \ldots,  b_{n-1}, b_n +1)$ as above. Since $\bpol(\sm) \in \bpol(I:x_n)$, we have $\bpol(\sm) \in \wfm^\bb$ for all $\bb \in \Psi(\ba)$.
Recall the description \eqref{Psi(a)} of $\Psi(\ba)$.  
If $x_n$ does not divide $\sm$, there exists some $1 \le i \le n-1$ such that $x_{i, b_i} \, | \, \bpol(\sm)$. Hence $\bpol(\sm) \in \wfm^{\wba}$. If $x_n$  divides $\sm$, then it can be possible that $x_{i, b_i}$ does not divide  $\bpol(\sm)$ for any $1 \le i \le n-1$. 
Note that  $\sm/x_n \in I:x_n$ and $\bpol(\sm/x_n ) \in \wfm^\bb$ for all $\bb \in \Psi(\ba)$. Hence, we have $x_{n,b_n} \, | \, \bpol(\sm/x_n )$ in this case.  It implies that $x_{n, b_n+1} \, | \,  \bpol(\sm)$, and hence $\bpol(\sm) \in \wfm^{\wba}$. 

Next, take an arbitrary $\ba \in G \setminus \varphi(F)$, and set $\oba= (b_1, \ldots,  b_{n-1}, b_{n-1})$ as above. Set $e:= \deg_{x_n} (\sm)$, where $\deg_{x_i}(-)$ stands for the degree with respect to the variable $x_i$. Then $\sn:= \sm \cdot (x_{n-1}/x_n)^e \in \oI$, and hence $\bpol(\sn) \in \bpol(\oI) \subset \wfm^\bb$ for $\bb:= (b_1, \ldots,  b_{n-1}) \in \Psi (\ba) $. It follows that $\bpol(\sm) \in \wfm^{\oba}$. In fact, if $x_{i,b_i} \, | \, \bpol(\sn)$ for some $i <n-1$, then  $x_{i,b_i} \, | \, \bpol(\sm)$. 
If $x_{n-1, b_{n-1}} \, | \,  \bpol(\sn)$, then either  $x_{n-1, b_{n-1}}$ or $x_{n, b_{n-1}}$ divides $ \bpol(\sm)$. Now we have shown \eqref{ETS2}.

Next, we will prove the inclusion $\supset$ of \eqref{ETS}. To do this, it suffices to 
show that 
$$
\bpol(\sm) \not \in \Bigl(\bigcap_{\ba \in F_1} \wfm^{\wba} \Bigr )   \cap  \Bigl(\bigcap_{\ba \in G \setminus \varphi(F)} \wfm^{\oba} \Bigr)
$$
for $\sm \in G(I:x_n) \setminus I$. 
Since $\sm \not \in I$, there is some $\ba \in F_1$ with $\sm \not \in \fm^{\ba + \be_n}$, 
or some $\ba \in G \setminus \varphi(F)$ with $\sm \not \in \fm^{\ba \oplus \be_n}$. 
If $\sm \not \in \sm^{\ba+\be_n}$, then $x_{i,j} | \bpol(\sm)$ implies $j  \le \sum_{k=1}^i \deg_{x_k}(\sm) \le \sum_{k=1}^i (a_k-1) = \Bigl ( \sum_{k=1}^i a_k \Bigr ) -i = b_i-1$ for $i \le n-1$, and $j \le b_n$ for $i=n$. It means that $\bpol(\sm) \not \in \wfm^{\wba}$.  Similarly, $\sm \not \in \fm^{\ba \oplus \be_n}$ implies $\bpol(\sm) \not \in \wfm^{\oba}$. 
Now we have shown that \eqref{irred main eq} holds

It remains to show that there is no inclusion among ideals $\wfm^\bb$ for $\bb \in \Psi(E)$, but this is easy.
\qed

\begin{ex}\label{submain example}
Consider a strongly stable ideal $I = (x_{1}^2, x_{1}x_{2}, x_{1}x_{3}, x_{2}^2 , x_{2}x_{3}^2)$, which is a slight modification of the one  in Example~\ref{main example}. From the irreducible decomposition  $I= (x_1, x_2) \cap (x_1^2, x_2,x_3) \cap (x_1,x_2^2,x_3^2) $, let us construct the decomposition  of $\bpol(I)$.  Theorem~\ref{irred main} states that  
$(x_1, x_2)$ yields $(x_{1,1}, x_{2,1})$, $(x_1^2, x_2,x_3)$ yields $(x_{1,2}, x_{2,2},x_{3,2})$, but 
$(x_1,x_2^2,x_3^2)$ yields  $(x_{1,1},x_{2,2},x_{3,2})$ and $(x_{1,1},x_{2,2},x_{3,3})$. Now we get the irreducible decomposition 
$$\bpol(I)= (x_{1,1}, x_{2,1}) \cap (x_{1,2}, x_{2,2},x_{3,2}) \cap (x_{1,1},x_{2,2},x_{3,2}) \cap(x_{1,1},x_{2,2},x_{3,3}).$$
\end{ex}

The next result concerns the {\it arithmetic degree} $\adeg(S/I)$ of $S/I$. For the basics of this notion, consult \cite[\S1]{V}.  However, following \cite{MVY}, we use the refinement $\adeg_i(S/I)$ of $\adeg(S/I)$ for $0 \le i \le \dim S/I$, which measures the contribution of the dimension $i$ components of $I$. Hence $\adeg(S/I)=\sum_{i \ge 0} \adeg_i(S/I)$.  

\begin{cor}\label{adeg}
Let $I$ be a strongly stable ideal with the irreducible decomposition  \eqref{irred decomp}. 
For $(a_1, \ldots, a_t) \in E$ (recall that $a_t > 0$), set $t(\ba):=t$ and $e(\ba):=a_t$. 
Then we have 
$$\adeg_i(S/I)=\sum_{\substack{ \ba \in E \\ t(\ba)=n-i}} e(\ba)$$
for each $i$. Hence, 
$$\adeg(S/I)=\sum_{\ba \in E} e(\ba) $$
and 
$$\deg(S/I)=\sum_{\substack{ \ba \in E \\ t(\ba)=\operatorname{ht}(I)}} e(\ba).$$
\end{cor}

\begin{proof}
Set $\wI :=\bpol(I)$. By an argument similar to the proof of Theorem~\ref{I^sigma}, we have  
\begin{eqnarray*}
\adeg_{n-c}(S/I) &=&\deg (\Ext^c_S(S/I,S)) \\
&=&\deg (\Ext^c_{\wS}(\wS/\wI,\wS))\\
&=& \text{the number of codimension  $c$ associated primes of $\wI$}.  
\end{eqnarray*}
Take $\ba \in E$ with $t(\ba)=c$. Then $\ba$ yields $e(\ba)$ irreducible components of $\wI$ of  codimension $c$.  Any codimension $c$ component of $\wI$ is given in this way,  and they are all distinct. So we are done. 
\end{proof}

\begin{cor}\label{components and lc2}
Let $I$ be a strongly stable ideal with the irreducible decomposition \eqref{irred decomp}. 
For $\ba=(a_1, \ldots, a_t) \in E$ (recall that $a_t > 0$), set $t(\ba):=t$, $w(\ba):=n- \sum_{i=1}^t a_i$, and $e(\ba):=a_t$. 
Then the Hilbert series of the local cohomology module $H_\fm^i(S/I)$ is given by 
$$ H(H_{\fm}^i (S/I), \lambda^{-1})
%&=&\frac{\sum_{j\in \ZZ} \# \{ \ba \in E \mid t(\ba)= n-i, \, w(\ba) - e(\ba) < j \le w(\ba)\}\cdot t^{i-j+1}}{(1-\lambda)^i}\\
= \left(\sum_{\substack{\ba \in E, \\ t(\ba)=n-i}} (\lambda^{w(\ba)} + \lambda^{w(\ba)+1}+ \cdots+ \lambda^{w(\ba)+e(\ba)-1})\right)  /(1-\lambda)^i. $$
\end{cor}

\begin{proof}
For $\ba=(a_1, \ldots, a_t) \in E$ with $|\ba|:=\sum_{i=1}^t a_i$,  $\Psi(\ba)$ is the  set 
$$\{ \, (b_1, b_2, \ldots, b_{t-1}, c) \mid |\ba| - t-a_t +2 \le c \le  |\ba| - t+1 \, \}$$
with $a_t$ elements. Here $b_i =(\sum_{j=1}^i a_j)-i+1$ for each $i$, while this value is not important now. 
By Theorem~\ref{irred main}, for $\bb \in \Psi(\ba)$, $\wfm^\bb$ is an irreducible component of $\bpol(I)$, and  any irreducible component
is given in this way. 

As we have shown in the proof of Theorem~\ref{localcoh}, 
for  $\bb= (b_1, b_2, \ldots, b_{t-1}, c)  \in \Psi(\ba)$, the component $\wfm^\bb$ contributes  to the Hilbert series of $H_\fm^i(S/I)$ if and only if $i=n-t$. If $i=n-t$, the contribution is $\lambda^{i-c+1}/(1-\lambda)^i$. Here, the numerator equals $\lambda^{n-t-c+1}$, and the exponent $n-t-c+1$ moves in the range 
\begin{eqnarray*}
n-t- ( |\ba| - t+1 )+1&  \le  \ n-t-c+1 \   \le & n-t -(|\ba| - t-a_t +2)+1 \\
w(\ba)  & \le \ n-t-c+1 \  \le & w(\ba) + e(\ba)-1. 
 \end{eqnarray*} 
Hence the contribution of $\ba \in E$ to $ H(H_{\fm}^i (S/I), \lambda^{-1})$ is 
$$
\begin{cases}
0 & \text{if $i\ne  n- t(\ba)$}, \\ 
\dfrac{\lambda^{w(\ba)} + \lambda^{w(\ba)+1}+ \cdots+ \lambda^{w(\ba)+e(\ba)-1}}{(1-\lambda)^i} & \text{if $i =  n- t(\ba)$.} 
\end{cases}
$$
So we are done. 
\end{proof}

\begin{ex}
This is a continuation of Example~\ref{submain example}. For  the strongly stable ideal $I= (x_1, x_2) \cap (x_1^2, x_2,x_3) \cap (x_1,x_2^2,x_3^2)\subset \kk[x_1, x_2, x_3]$, let $\ba$ and $\bb$ denote the exponent vectors $(2,1,1)$ and $(1,2,2)$ of the height 3 components, respectively.  
With the notation of Corollaries~\ref{adeg} and \ref{components and lc2}, we have $w(\ba)=-1, e(\ba)=1, w(\bb)=-2$ and $e(\bb)=2$. Hence  we have $\adeg_0(S/I) =e(\ba)+e(\bb)=3$.  Similarly,  $H(H_{\fm}^0 (S/I),\lambda^{-1}) = \lambda^{-2} +2 \lambda^{-1}$, where the contributions of the components $(x_1^2, x_2,x_3)$ and $(x_1,x_2^2,x_3^2)$ are $\lambda^{-1}$ and $\lambda^{-2}+\lambda^{-1}$,  respectively. 
\end{ex}

\section{Remarks on irreducible components of strongly stable ideals}
For  $\ba = (a_1, \ldots, a_t)\in \NN^t$  with $a_t>0$ and $\bb = (b_1, \ldots, b_t) \in (\ZZ_{>0})^t$, we set 
$$\hba:=(a_1+1, a_2+1, \ldots, a_{t-1}+1, a_t) \in (\ZZ_{>0})^t, $$ and  
$$\wbb:=(b_1-1, b_2-1, \ldots, b_{t-1}-1, b_t) \in \NN^t.$$
Note that this notation is different from that in the previous section. 

For a monomial $x^\ba \in S$ with $\ba=(a_1, \ldots, a_n) \in \NN^n$, recall that  $\nu(x^\ba) = \max \{ \, i \mid  a_i > 0 \,  \}$. For a monomial ideal $I$, set 
$\nu(I) := \max \{\, \nu(x^\ba) \mid x^\ba \in G(I) \, \}$. If $I$ is strongly stable, then it is well-known that 
\begin{eqnarray*}
\nu(I) &=& \max \{\, \operatorname{ht}(\fm^\ba) \mid \text{$\fm^\ba$ is an irreducible component of $I$ } \}\\
&=& \max \{\, l \mid \text{$\fm^\ba$ is an irreducible component of $I$ for some $\ba \in (\ZZ_{>0})^l$ } \}. 
 \end{eqnarray*}

\begin{prop}\label{gene}
Let $I\subset S$ be a strongly stable ideal with $l:=\nu(I)$. For $x^\ba\in G(I)$ with $\nu(x^\ba)=l$, $\fm^{\hba}$ is an irreducible component of $I$. Conversely, any irreducible component of height $l$ arises in this way.  More precisely, if $\fm^\bb$ is an irreducible component of $I$ with $\mathrm{ht}(\fm^\bb)=l$ (in other words, $(\bb\in \ZZ_{>0})^l$), then $x^{\wbb} \in G(I)$.
\end{prop}

\begin{proof}
Take $x^\ba\in G(I)$ with $\nu(x^\ba)=l$. Since $x^\ba/x_l \notin I$, there is an irreducible component $\fm^\bb$  with $x^\ba/x_l\notin \fm^\bb$. Clearly, $\bb \in (\ZZ_{>0})^l$ now. We will show that $\bb=\hba$. Since $x^\ba \in I \subset \fm^\bb$ and $x^\ba/x_l\notin \fm^\bb$, we have 
$b_i >a_i$  for all $i \le l-1$, and $a_l   =b_l $.  
If $b_i>a_i+1$ for some $i \le l-1$, then $(x_i/x_l) \cdot x^\ba \notin \fm^\bb$, and this is a contradiction. Therefore, $b_i= a_i+1$ holds for all  $i\le l-1$, and we have $\bb=\hba$.

Conversely, we assume that $\fm^\bb$ is an irreducible component of $I$ with $\mathrm{ht}(\fm^\bb)=l$.
First, we will show that $\sm := x^{\wbb}\in I$. For a contradiction, assume that $\sm \notin I$. Then there is an irreducible component $\fm^\bc$ of $I$ with $\sm \notin \fm^\bc$. 
Then we have $c_i \ge b_i$  for all $i < l$, and $c_l > b_l$  (if $\bc \in (\ZZ_{>0})^l$).  
It follows that $\fm^\bc \subsetneq \fm^\bb$. This is a contradiction.

Next we will show that $\sm\in G(I)$. 
Since we have shown that $\sm \in I$, there is $x^\bc \in G(I)$ which divides $\sm$. Clearly, $c_i \le b_i -1$ for all $i \le l-1$, $c_l \le b_l$, and $c_i =0$ for all $i >l$. 
Since $x^\bc \in \fm^\bb$, we have $\nu(\fm^\bc)=l$ and $c_l=b_l>0$. 
Moreover, since $(x_i/x_l)\cdot x^\bc \in \fm^\bb$ for all $i <l$, we have $c_i = b_i -1$ for all $i \le l-1$. Hence we have $\wbb =\bc$, and $\sm = x^{\wbb} =x^\bc \in G(I)$. 
\end{proof}

\begin{cor}
Let $I\subset S$ be a strongly stable ideal, and set $l:=\nu(I)$.  Then we have 
$$\adeg_{n-l}(S/I)=\sum_{\substack{x^\ba\in G(I)\\ \nu(x^\ba)=l}} a_l$$
\end{cor}
\begin{proof}
By Corollary \ref{adeg} and Proposition \ref{gene}, the assertion follows.
\end{proof}

\begin{rem}
Let $I\subset S$ be a strongly stable ideal with $l:=\nu(I)$. If $S/I$ is Cohen--Macaulay (equivalently, $l= \operatorname{ht}(I)$), 
then Proposition~\ref{gene} directly gives the irreducible decomposition of $I$. 
If  $l > \operatorname{ht}(I)$, then we consider 
$$I:x_l^{\infty}:=\{ \, f \in S \mid \text{$x_l^q f\in I$ for $q \gg 0$ } \}.$$ This  is a strongly stable ideal again, and the intersection of the irreducible components of $I$ whose heights are less than $l$.  Moreover, $G(I:x_l^{\infty})$ can be easily computed from $G(I)$. 
Therefore, combining this operation with Proposition \ref{gene}, we can compute the  irreducible decomposition of $I$.
\end{rem}

\begin{ex}
For the strongly stable ideal $I=(x_1^3,x_1^2x_2, x_1x_2^2, x_1x_2x_3^2, x_1^2x_3^2)$, the generators $x_1x_2x_3^2$ and  $x_1^2x_3^2$ yield  
$(x_1^2,x_2^2,x_3^2)$ and $(x_1^3,x_2,x_3^2)$, respectively.  Next, consider the strongly stable ideal $I':=I:x_3^{\infty}=(x_1^2, x_1x_2)$, and it has an irreducible component $(x_1^2, x_2)$ given by $x_1x_2$, which is also an irreducible component of $I$ itself. 
Finally, $I': x_2^\infty =(x_1)$ itself is  an irreducible component of $I$. Hence the irreducible decomposition of $I$ is 
$$I=(x_1)\cap (x_1^2,x_2)\cap (x_1^2,x_2^2,x_3^2) \cap (x_1^3,x_2,x_3^2).$$
\end{ex}

%For an irreducible monomial ideal $Q=\fm^\ba$ with $\ba=(a_1,\ldots, a_t) \in (\ZZ_{>0})^t$ for some $t \le n$, set $Q_{i,j}:=\fm^{\ba-\be_i+\be_j}$ for $i,j\in [t]$ with $i<j$, $a_i>1$.

\begin{thm}\label{right shift}
Let $I \subset S$ be a monomial ideal with the irredundant irreducible decomposition \eqref{irred main eq}.  (Note that the irreducible decomposition of a strongly stable ideal is always in this form.) Then the following are equivalent. 
\begin{itemize}
\item[(1)] $I$ is strongly stable. 
\item[(2)] If $\ba=(a_1, \ldots, a_t) \in E \cap (\ZZ_{>0})^t$,  $a_i >1$ and $i < j \le t$, then there is some $\bb \in E$ such that $\fm^{\ba-\be_i+\be_j} \supset \fm^\bb$, where  $\be_{i}\in \NN^t$ is the $i$-th unit vector. 
\end{itemize}
\end{thm}

\begin{proof}
(1)  $\Rightarrow$ (2): For a contradiction, assume that  a strongly stable ideal $I$ does not satisfy (2). Then, for each $\bb \in E$, we have $\fm^{\ba-\be_i+\be_j} \not \supset \fm^\bb$, and we can take a monomial $\sm_\bb \in G(\fm^\bb)$ with $\sm_\bb \not \in \fm^{\ba-\be_i+\be_j}$ (of course, $\sm_\bb=x_k^{b_k}$ for some $k \in [n]$). 
% It is easy to see that $\sm_\ba= x_j^{a_j}$. 
Let $\sm$ be the least common multiple of $\{ \sm_\bb \mid \bb \in E \}$.    
Since $\sm \in \bigcap_{\bb \in E} \fm^\bb =I \subset \fm^\ba$ and $\sm \notin \fm^{\ba-\be_i+\be_j}$, the degree $\deg_{x_k} (\sm)$ with respect to $x_k$ is 
$$\begin{cases}   < a_k  \, \, \, \, \,  ({\rm if}\, \, k \neq i,j),  \\  =a_j    \, \, \, \, \,  ({\rm if}\, \, k=j), \\  <a_i-1    \, \, \, \, \,  ({\rm if}\, \, k =i).  \end{cases}$$
So we have $(x_i/x_j) \cdot \sm \notin \sm^\ba$, and hence $(x_i/x_j) \cdot \sm \notin I$.  It contradicts the assumption that $I$ is strongly stable and $\sm \in I$. 

(2)  $\Rightarrow$ (1): For a contradiction, we assume that $I$ satisfies (2) but it is not strongly stable. 
Then there are some $\sm \in  G(I)$ and some $i \ge 2$ such that  $x_i$ divides  $\sm $ and  $(x_{i-1}/x_i) \cdot \sm \notin \fm^\ba$ for some $\ba =(a_1, \ldots, a_t) \in E$. Then it is easy to see that $a_{i-1}>1$ and $t \ge i$.  
By (2), we have  $\fm^{\ba -\be_{i-1}+\be_i} \supset \fm^\bb$ for some $\bb \in E$.  
 Since $(x_{i-1}/x_i) \cdot \sm \notin \fm^\ba$, we have  $\sm \notin \fm^{\ba -\be_{i-1}+\be_i} $. It contradicts that  $\sm \in I \subset \fm^\bb$.
\end{proof}

\begin{ex}
For a strongly stable ideal  $I = (x_1^2, x_1x_2, x_2^3, x_1x_3, x_2^2x_3)\subset \kk[x_1,x_2,x_3]$, we have the irreducible decomposition 
$$
I=(x_1,x_2^2)\cap(x_1^2,x_2,x_3)\cap(x_1,x_2^3, x_3). 
$$
We consider the irreducible component $\fm^\ba=(x_1,x_2^3, x_3)$ with $\ba = (1,3,1)$. Clearly, 
$$\fm^{\ba-\be_2+\be_3}=\fm^{(1,2,2)}=(x_1,x_2^2, x_3^2) \supset(x_1,x_2^2),$$ where $(x_1,x_2^2)$ is an irreducible component.

Next consider the ideal $J= (x_1^3,x_1^2x_2,x_1x_2^2,x_2^3,x_1x_3)$ with the irreducible decomposition  
$$J=(x_1,x_2^3)\cap (x_1^2,x_2^2,x_3) \cap (x_1^3,x_2,x_3).$$
For the irreducible component $\fm^{\ba}$ with $\ba=(2,2,1)$,  we have 
$$\fm^{\ba-\be_2+\be_3} =\fm^{(2,1,2)} =(x_1^2, x_2,x_3^2) \nsupseteq (x_1,x_2^3),  (x_1^3,x_2,x_3),$$ and $J$ is not strongly stable.
Of course, we can check this directly. In fact, we have $\sm:= x_1x_3 \in J$, but $(x_2/x_3) \cdot \sm = x_1x_2 \notin J$. 
\end{ex}

\section*{Acknowledgments}
We are grateful to Professor Yuji Yoshino for  stimulating  discussion and valuable comments.  
We also thank Professor Gunnar Fl$\o$ystad for telling us about his paper \cite{F}.

\end{document}